\documentclass[12pt]{amsart}
\usepackage{amssymb,latexsym,amsmath,amscd,mathrsfs,yfonts,hyperref}
\usepackage{fullpage}
\input xy
\xyoption{all}

\newcommand{\RR}{\mathbb R}
\newcommand{\ZZ}{\mathbb Z}

\newcommand{\CC}{\mathbb C}

\newcommand{\NN}{\mathbb N}

\newcommand{\cpt}{\mathbb K}

\def\bK{\mathbf K}

\def\Perf{\mathtt{Perf}}

\def\KKcat{\mathtt{KK}}

\def\ev{\textup{ev}}
\def\od{\textup{od}}

\def\id{\mathrm{id}}

\def\End{\textup{End}}

\def\Hom{\textup{Hom}}

\def\op{\textup{op}}

\def\prot{\hat{\otimes}}

\def\rep{\textup{rep}}

\def\sC{\text{$\sigma$-$C^*$}}

\def\Hmo{\mathtt{Hmo}}

\def\K{\textup{K}}

\def\Sps{\mathtt{Sp^\Sigma}}

\def\Sp{\mathtt{Sp}}
\def\hSp{\mathtt{hSp}}

\def\Kw{\mathbf{K}_\Sigma^{\textup{top}}}

\def\cE{\mathcal{E}}
\def\cF{\mathcal{F}}
\def\cS{\mathcal S}

\def\DGcorr{\mathtt{NCC_{dg}}}

\def\pt{\textup{pt}}

\def\DGcat{\mathtt{DGcat}}
\def\CAlg{\mathtt{{C^*}}}

\def\KK{\textup{KK}}
\def\E{\textup{E}}

\def\H{\textup{H}}
\def\End{\textup{End}}

\def\cC{\mathcal C}

\def\1{\bf{1}}

\def\Csep{\mathtt{SC^*}}

\def\Map{\textup{Map}}
\def\cD{\mathcal D}

\def\TT{\mathbb{T}}

\def\CT{\textup{CT}}

\def\A{\textup{A}}

\def\cA{\mathcal{A}}
\def\wK{\mathbf{K}^{\textup{w}}}
\def\cB{\mathcal{B}}

\newcommand{\map}{\rightarrow}
\newcommand{\functor}{\rightarrow}

\def\h{\mathtt{h}}

\def\KT{\textup{K}^{\textup{top}}}

\def\ua{\mathcal{U}_A}

\def\Oinf{\mathcal{O}_\infty}
\def\Th{\mathbf{Th}}
\def\hpf{\mathtt{HPf_{dg}}}
\def\mot{\mathtt{Mot^{loc}_{dg}}}
\def\ul{\mathcal{U}^{\mathtt{loc}}_{\mathtt{dg}}}
\def\one{\mathbf{1}}

\newcommand{\beq}{\begin{eqnarray}}
\newcommand{\beqn}{\begin{eqnarray*}}
\newcommand{\eeq}{\end{eqnarray}}
\newcommand{\eeqn}{\end{eqnarray*}}

\theoremstyle{definition}
\newtheorem{thm}{Theorem}[section]
\theoremstyle{definition}

\newtheorem*{Rem}{Remark}

\newtheorem{lem}[thm]{Lemma}
\newtheorem{prop}[thm]{Proposition}
\newtheorem{cor}[thm]{Corollary}
\newtheorem{ex}[thm]{Example}
\newtheorem{defn}[thm]{Definition}
\newtheorem{rem}[thm]{Remark}

\begin{document}

\title{Algebraic $\K$-theory, $\K$-regularity, and $\TT$-duality of $\Oinf$-stable $C^*$-algebras}
\author{Snigdhayan Mahanta}
\email{snigdhayan.mahanta@mathematik.uni-regensburg.de}
\address{Fakult{\"a}t f{\"u}r Mathematik, Universit{\"a}t Regensburg, 93040 Regensburg, Germany.}
\thanks{This research was supported by the Deutsche Forschungsgemeinschaft (SFB 878 and SFB 1085), ERC through AdG 267079, and the Humboldt Professorship of M. Weiss.}
\dedicatory{Dedicated to Professor Marc A. Rieffel on the occasion of his 75th birthday.}

\begin{abstract}
We develop an algebraic formalism for topological $\TT$-duality. More precisely, we show that topological $\TT$-duality actually induces an isomorphism between noncommutative motives that in turn implements the well-known isomorphism between twisted $\K$-theories (up to a shift). In order to establish this result we model topological $\K$-theory by algebraic $\K$-theory. We also construct an $\E_\infty$-operad starting from any strongly self-absorbing $C^*$-algebra $\cD$. Then we show that there is a functorial topological $\K$-theory symmetric spectrum construction $\Kw(-)$ on the category of separable $C^*$-algebras, such that $\Kw(\cD)$ is an algebra over this operad; moreover, $\Kw(A\prot\cD)$ is a module over this algebra. Along the way we obtain a new symmetric spectra valued functorial model for the (connective) topological $\K$-theory of $C^*$-algebras. We also show that $\Oinf$-stable $C^*$-algebras are $\K$-regular providing evidence for a conjecture of Rosenberg. We conclude with an explicit description of the algebraic $\K$-theory of $ax+b$-semigroup $C^*$-algebras coming from number theory and that of $\Oinf$-stabilized noncommutative tori.
\end{abstract}

\maketitle

\begin{center}
{\bf Introduction}
\end{center}

Within the category of separable $C^*$-algebras $\Csep$ the {\em stable} ones, i.e., those $A\in\Csep$ satisfying $A\prot\cpt\cong A$, play a privileged role. For instance, it is known that they satisfy the Karoubi conjecture and appear very naturally in the context of twisted $\K$-theory. The results in this article demonstrate that {\em $\Oinf$-stable} separable $C^*$-algebras, i.e., those $A\in\Csep$ satisfying $A\prot\Oinf\cong A$, deserve a similar prominent status. Moreover, the {\em Cuntz algebra $\Oinf$} is {\em strongly self-absorbing}, which has several interesting ramifications. 

Ever since its inception by Kontsevich \cite{HMS} the homological mirror symmetry conjecture has promoted rich interaction between geometry, algebra, and (higher) category theory. Mirror symmetry is related to $\TT$-duality via the Strominger--Yau--Zaslow conjecture \cite{SYZ} and hence we believe that it is worthwhile to have an algebraic formalism for $\TT$-duality at our disposal. One aspect of this theory is the Bunke--Schick {\em topological $\TT$-duality}, which has received a lot of attention in the mathematical literature. It is insensitive to subtle geometric structures but its mathematical underpinnings are very well understood \cite{BunSch1,BunRumSch,BunSchSpiTho}. One of the objectives of this article is to develop an algebraic formalism for topological $\TT$-duality relating it to the theory of noncommutative motives \cite{KonNotes,KonMot,TabGarden,MarTabApplications}. Along the way we obtain several interesting applications to algebraic $\K$-theory and $\K$-regularity of $C^*$-algebras. The novelty of our approach lies in the use of the Cuntz algebra $\Oinf$ and the construction of certain operadic actions on (twisted) $\K$-theory.

Let us briefly describe our main results. Given any $C^*$-algebra $A$ we can functorially associate its noncommutative motive $\hpf(A)$ with it \cite{KQ}. As a mathematical object $\hpf(A)$ is a differential graded category that is defined purely algebraically. In Section \ref{Duality} we show that if $A$ and $A'$ are $\KK$-equivalent separable $C^*$-algebras, then $\hpf(A\prot\Oinf)$ and $\hpf(A'\prot\Oinf)$ are isomorphic objects in the category of noncommutative motives (cf. Theorem \ref{OMotive} and Corollary \ref{KKfac}). We also show that the nonconnective $\K$-theory of the noncommutative motive $\hpf(A\prot\Oinf)$ is naturally isomorphic to the topological $\K$-theory of $A$ (cf. Theorem \ref{OK}). It is known that under favourable circumstances topological $\TT$-duality can be expressed as a $\KK$-equivalence between two separable $C^*$-algebras \cite{BMRS,BMRS2}. Thus our results show that in such cases one actually has an isomorphism of noncommutative motives that implements the well-known isomorphism (up to a shift) between the twisted $\K$-theories. Since noncommutative motives constitute the universal cohomology theory of noncommutative spaces, our results demonstrate that topological $\TT$-duality implements an isomorphism of universal cohomology theories. The treatment here is completely algebraic; we model topological $\K$-theory via algebraic $\K$-theory following (and somewhat refining) our earlier approach in \cite{KQ}.

Let us now explain the significance of $\Oinf$ in this context that made no appearance in \cite{KQ}. It is a noteworthy example of a strongly self-absorbing $C^*$-algebra \cite{TomWin} with very interesting structural properties. Using a result from \cite{CorPhi} (see also \cite{KarWod}) one can deduce that nonconnective algebraic $\K$-theory agrees naturally with topological $\K$-theory for $\Oinf$-stable $C^*$-algebras. There has been considerable interest in associating symmetric spectra with $C^*$-algebras functorially, whose homotopy groups are the topological $\K$-theory groups. Various algebraic $\K$-theory {\em machines} (like Waldhausen $\K$-theory) canonically produce symmetric spectra (see Remark 1.2.6 of \cite{HSS}). Combining this fact with our techniques we obtain a new functorial symmmetric spectra valued model for the (connective) topological $\K$-theory of $C^*$-algebras (cf. Theorem \ref{GenHom} and Remark \ref{HSS}). This algebraic method completely circumvents the analytical difficulties that one needs to overcome in a direct approach like the one in \cite{JoaKHom}. Applying the technology of \cite{KQ} this result could have been deduced without invoking $\Oinf$. However, we exhibit more algebraic structure on $\K$-theory using our current formalism. Indeed, we construct for every strongly self-absorbing $C^*$-algebra $\cD$ an {\em $\E_\infty$-operad} that we call the {\em strongly self-absorbing $\cD$-operad} (cf. Definition \ref{D} and Proposition \ref{Einfinity}). Then we show that there is a symmetric spectra valued model for the topological $\K$-theory of $\cD$ that is an algebra over the $\cD$-operad; moreover, the topological $\K$-theory symmetric spectrum of any $\cD$-stable separable $C^*$-algebra is a module over this algebra (cf. Theorem \ref{Operad} for a more general formulation and also Example \ref{OpAction}). These operadic structures up to coherent homotopy in symmetric spectra can be further rectified to strict ones (see, for instance, \cite{MMSS,ElmMan}). For the $\infty$-categorical counterparts of related results the readers may refer to \cite{MyColoc}.

Using similar ideas as before we show that $A\prot\Oinf$ is $\K$-regular for any $C^*$-algebra $A$ (cf. Theorem \ref{Kreg}), providing evidence for a conjecture of Rosenberg \cite{RosComparison}. We carry out an explicit computation of the algebraic $\K$-theory of $ax+b$-semigroup $C^*$-algebras associated with number rings \cite{CunDenLac,LiSGC} (cf. Theorem \ref{ax+b}). Noncommutative tori constitute arguably the most widely studied class of noncommutative spaces. Their geometric invariants were studied extensively by Connes and Rieffel (see, for instance, \cite{ConBook,ConRie,RieNCTori}). We show that the algebraic $\K$-theory of noncommutative tori are explicitly computable after $\Oinf$-stabilization (cf. Theorem \ref{NCTori}). Using a powerful result of Rieffel \cite{RieNCTori}, one also obtains a clear understanding of the elements of the algebraic $\K$-theory groups in low degrees (see Remark \ref{RieElt}).

\begin{Rem}
Some of the arguments below exploit a cute trick (cf. Lemma \ref{Basic}). The range of applicability of this trick is much broader than the case explored here (see, for instance, Proposition 1.1.2 of \cite{RorClass}). The author is grateful to D. Enders for pointing out that Proposition \ref{CorPhi}, that uses this trick, can be generalized to all $C^*$-algebras of the form $A\prot B$ with $B$ {\em properly infinite}. We encourage the readers to consult \cite{CorPhi} for an even more general result.
\end{Rem}

\medskip
\noindent
{\bf Notations and Conventions:} In the sequel we denote the category of all (resp. separable) $C^*$-algebras by $\CAlg$ (resp. $\Csep$). We denote by $\bK(-)$ [resp. $\K(-)$] the nonconnective algebraic $\K$-theory spectrum [resp. $\K$-theory group] functor and by $\bK^{\textup{top}}(-)$ [resp. $\KT(-)$] the (twisted) topological $\K$-theory spectrum [resp. $\K$-theory group] functor on $\CAlg$. Unless otherwise stated, all spaces are assumed to be Hausdorff and $\prot$ will denote the maximal $C^*$-tensor product (expect in Example \ref{OpAction}). 

\medskip
\noindent
{\bf Acknowledgements:} The author would like to thank J. Cuntz, D. Enders, G. Horel, J. Lind, and A. Thom for beneficial discussions. The author is also grateful to U. Bunke, N. C. Phillips, and J. Rosenberg for constructive feedback. Careful reviews resulting in several corrections suggested by the anonymous referees have helped improve the exposition significantly. A part of this project was also supported by the fellowship for research scholars of the Max Planck Institute for Mathematics, Bonn.

\section{Topological $\TT$-duality and noncommutative motives} \label{Duality}
For the benefit of the reader we briefly discuss topological $\TT$-duality and noncommutative motives as well as their $\K$-theory before explaining our results.

\subsection{Topological $\TT$-duality} $\TT$-duality is an interesting phenomenon is string theory, some of whose mathematical aspects were studied in \cite{BEM2}. We are solely going to focus on axiomatic topological $\TT$-duality from \cite{BunSch1}, which builds upon the earlier work in \cite{BEM2}. Let $B$ be a topological base space. Consider the category of {\em pairs} $(E,h)$, where $\pi:E\map B$ is a principal $S^1$-bundle over $B$ and $h\in \H^3(E,\ZZ)$. Two such pairs $(E_1,h_1)$ and $(E_2,h_2)$ are isomorphic if there is an isomorphism $F:E_1\map E_2$ of principal bundles such that $F^*h_2 =h_1$. Two pairs $(E_1,h_1)$ and $(E_2,h_2)$ are said to be {\em $\TT$-dual} if there is a Thom class $\Th\in\H^3(S(V),\ZZ)$ for $S(V)$ such that $h_1=i_1^* \Th$ and $h_2= i_2^*\Th$. Here $S(V)$ is the sphere bundle of $V:= E_1\times_{S^1}\CC \oplus E_2\times_{S^1}\CC$ and $i_k: E_k\map S(V)$ are the canonical maps for $k=1,2$. This definition implies the following correspondence picture: Let $\pi_k: E_k\map B$ with $k=1,2$ be two principal $S^1$-bundles and $(E_1,h_1)$ and $(E_2,h_2)$ be $\TT$-dual pairs. Then there is a commutative diagram \beq
\xymatrix{
& E_1\times_B E_2 \ar[ld]_{\textup{pr}_1}\ar[dd]^q \ar[rd]^{\textup{pr}_2} \\
E_1 \ar[rd]_{\pi_1} && E_2 \ar[ld]^{\pi_2}\\
& B,
}
\eeq such that $\textup{pr}_1^* (h_1) =\textup{pr}_2^* (h_2)$. This basic correspondence picture relates topological $\TT$-duality to cohomological quantization (see, for instance, \cite{Nuiten,UrsQuant}).

In \cite{BunSch1} Bunke--Schick showed that the association $B\mapsto \{\text{isom. classes of pairs over $B$}\}$ as a functor on topological spaces is representable. The representing space $\mathbf{E}$ supports a universal pair and any pair on $B$ can be obtained up to isomorphism via a pullback along some map $B\map \mathbf{E}$ (defined uniquely up to homotopy). Using the explicit construction of the universal object and the $\TT$-dual of the universal pair the authors were able to prove the existence and uniqueness of $\TT$-duality for $S^1$-bundles. One of the salient features of $\TT$-duality is the following: If $(E_1,h_1)$ and $(E_2,h_2)$ are $\TT$-dual pairs, then there is an isomorphism of twisted $\K$-theories: \beq \label{Twist}
(\KT)^\od(E_1,h_1)\simeq (\KT)^\ev(E_2,h_2) \quad \text{ and } \quad (\KT)^\ev(E_1,h_1)\simeq (\KT)^\od(E_2,h_2).
\eeq The theory of topological $\TT$-duality is not limited to $S^1$-bundles. However, for more general ($\prod_{i=1}^n S^1$)-bundles with $n>1$ the theory becomes quite subtle \cite{BunRumSch,BunSchSpiTho} and sometimes necessitates the use of $C^*$-algebras \cite{MatRos2}. Moreover, $C^*$-algebras appear quite naturally in the context of twisted $\K$-theory \cite{RosCT}. Thus it seems natural to study $\TT$-duality via $C^*$-algebras from the outset. The readers may refer to \cite{RosStringDuality} for a survey on the interactions between $C^*$-algebras, $\K$-theory, noncommutative geometry, and $\TT$-duality. Some recent results indicate that $\TT$-duality can even be related to Langlands duality \cite{DanErp,BunNik}.

\subsection{Noncommutative motives} Much like motives in algebraic geometry serve as the receptacle for the universal cohomology theory for algebraic varieties (with various interesting realization functors), noncommutative motives constitute the universal cohomology theory for noncommutative spaces in the sense of Kontsevich \cite{KonMot,KonNotes}. The rough idea behind its construction is to begin with a reasonable category (actually a model category) of noncommutative spaces and then enforce certain properties that are expected of any cohomology theory, i.e., pass to the suspension stabilization and implement appropriate versions of Morita invariance and localization. The techniques involved in this process are to some extent informed by both algebraic geometry and homotopy theory. The theory of noncommutative motives has interesting applications to $\K$-theory as well as a wide variety of other areas in mathematics \cite{KonNotes,MarTabApplications}. We briskly review its rudiments following \cite{KelDG,Tab3}. 

Let $k$ be a field of characteristic zero (for our purposes it is $\CC$). Let $\DGcat$ denote the category of (small) differential graded (DG) categories over $k$. In this setting a noncommutative space is the same thing as a DG category. There is a model category of modules (and even a subcategory of perfect modules) over a DG category $\cA$, whose homotopy category denoted by $D(\cA)$ is the derived category of $\cA$. We refer the readers to section 3.8 of \cite{KelDG}, where it is explained how a functor between DG categories $F:\cA\map\cB$ induces a restriction of scalars functor $D(\cB)\map D(\cA)$. A morphism of DG categories $F:\cA\map\cB$ is called a {\em derived Morita equivalence} or simply a {\em Morita morphism} if the restriction of scalars functor $D(\cB)\map D(\cA)$ is an equivalence of triangulated categories. The category $\DGcat$ supports a model structure, whose weak equivalences are the Morita morphisms, and its homotopy category is denoted by $\Hmo$. In the module category over a DG category one can perform (homotopical) analogues of most of the operations that are present on the module category over a ring. Let $\rep(A, B)\subset D(\cA^{\op}\otimes^{\mathbb{L}}\cB)$ be the full triangulated subcategory consisting of those $\cA$-$\cB$-bimodules $X$, such that $X(a,-) \in D(\Perf (\cB))$ for every object $a\in\cA$. Set $\Hmo_0$ to be the category whose objects are DG categories and whose morphisms are $\Hmo_0(\cA,\cB):=\K_0(\rep(\cA,\cB))$ with composition induced by the tensor product of bimodules. There is a functor $\ua: \DGcat\functor\Hmo_0$ that should be regarded as the {\em pure} motive associated with a noncommutative space. The functor $\ua$ is identity on objects and sends a morphism of DG categories to its class in the Grothendieck group of bimodules. The category of pure noncommutative motives $\Hmo_0$ further maps into a triangulated category of {\em mixed} noncommutative motive $\mot$. The composite functor $$\ul:\DGcat\overset{\ua}{\functor}\Hmo_0\functor\mot$$ has the property that it sends every exact sequence of DG categories to an exact triangle in $\mot$. Recall that a diagram of DG categories $\cA\map\cB\map\cC$ is called an {\em exact sequence} if the induced sequence $D(\cA)\map D(\cB)\map D(\cC)$ of triangulated categories is Verdier exact. The construction of $\mot$ uses the theory of Grothendieck derivators that we leave out from the discussion. In this article we are solely going to focus on the category of pure noncommutative motives $\Hmo_0$. There is a stable $\infty$-categorical counterpart of noncommutative motives \cite{BluGepTab} and our results relating topological $\TT$-duality with noncommutative motives also admit a generalization to the $\infty$-categorical setup (see Section 4 of \cite{MyNSHLoc}).

\subsection{Nonconnective $\K$-theory of pure noncommutative motives} \label{NegK}
Any pure noncommutative motive is merely an object $\cA\in\DGcat$, since the functor $\ua:\DGcat\map\Hmo_0$ is identity on objects. Using the corepresentability result (see Theorem 6.1 of \cite{TabGarden}) the quickest definition of the nonconnective algebraic $\K$-theory spectrum of a DG category $\cA$ is $\bK(\cA):=\RR\Hom(\ul(k),\ul(\cA))$. Note that the category of mixed noncommutative motives $\mot$ is canonically enriched over spectra; this is a consequence of the general formalism of triangulated derivators that is deployed to construct $\mot$. The nonconnective algebraic $\K$-theory groups of $\cA$ can be defined as $\K_n(\cA):= \Hom(\ul(k),\ul(\cA)[-n])$ for all $n\in\ZZ$. By construction the functor $\bK(-)$ (resp. $\K_n(-)$) factors through the category of pure noncommutative motives $\Hmo_0$.

The nonconnective algebraic $\K$-theory of a DG category $\cA$ can also be constructed by more traditional methods that we presently explain. The category of perfect cofibrant DG modules $\Perf(\cA)$ over $\cA$ admits the structure of a complicial exact category, whose weak equivalences are those maps that are isomorphisms in $D(\cA)$. Recall that a module over $\cA$ is perfect if and only if it is a compact object in the derived category $D(\cA)$ (see Corollary 3.7 of \cite{KelDG}). Applying the Waldhausen $\K$-theory functor $\wK(-)$ to the complicial exact category $\Perf(\cA)$ one gets the connective $\K$-theory spectrum $\wK(\cA)$ of the DG category $\cA$. From \cite{SchDeloop} (see also Section 3.2.33 of \cite{SchSedano}) one learns that there is a suspension construction $\Sigma(-)$ of a complicial exact category $\cE$, such that there is a natural homotopy equivalence \beq \label{bond}\wK(\cE)\overset{\sim}{\map} \Omega\wK(\Sigma(\cE)),\eeq provided the triangulated category associated with $\cE$ is idempotent complete (e.g., $\cE=\Perf(\cA)$). The nonconnective $\K$-theory spectrum of $\cA$, denoted by $\bK(\cA)$, is the spectrum whose $n$-th space is $\wK(\Sigma^n(\Perf(\cA)))$ and the structure maps are furnished by \eqref{bond}. Any unital $k$-algebra $R$ can be viewed as a DG category with one object $\bullet$, such that $\End(\bullet)=R$. In this case the above construction applied to $\cE=\Perf(R)$ recovers Quillen's higher algebraic $\K$-theory of $R$ in nonnegative degrees and Bass' negative $\K$-theory in negative degrees.

\subsection{Our results} We denote the category of separable $C^*$-algebras by $\Csep$ and the bivariant $\K$-theory category by $\KKcat$. There is a canonical functor $\iota:\Csep\functor\KKcat$, which is identity on objects and admits a universal characterization \cite{Hig1,CunKK}. Building upon an earlier work of Quillen \cite{QuiNonunitalK0} the author constructed a functorial passage $\hpf$ from separable $C^*$-algebras to (pure) noncommutative motives. For any $A\in\Csep$ the DG category $\hpf(A)$ consists of cochain complexes $X$ of right $\tilde{A}$-modules, where $\tilde{A}$ is the unitization of $A$, that satisfy \begin{enumerate}                                                                                                                                                                                                                                                                                                                                                                                                                                                                                                                                                                                           \item $X$ is homotopy equivalent to a strictly perfect complex (of right $\tilde{A}$-modules), and                                                                                                                                                                                                                                                                                                                                                                                                                                                                                                                                                                                         \item the quotient complex $X/XA$ is acyclic.                                                                                                                                                                                                                                                                                                                                                                                                                                                                                                                                                                                           \end{enumerate} The association $A\mapsto\hpf(A)$ gives rise to a functor $\Csep\functor\Hmo_0$ in an evident manner. The following two results (amongst others) are proved in \cite{KQ}:

\begin{thm} \label{Tdual}
There is a dashed functor below making the following diagram of categories commute (up to a natural isomorphism): \beqn
\xymatrix{
\Csep\ar[rr]^{A\mapsto A\prot\cpt}\ar[d]^\iota && \Csep\ar[d]^{\hpf}\\
\KKcat\ar@{-->}[rr] &&\Hmo_0.
}
\eeqn
\end{thm}

\begin{thm} \label{NCK}
 For any $A\in\Csep$ the homotopy groups of the nonconnective $\K$-theory spectrum of $\hpf(A\prot\cpt)$ are naturally isomorphic to the topological $\K$-theory groups of $A$.
\end{thm}

\begin{rem}
 In \cite{KQ} the author phrased the results in terms of $\DGcorr$, which was called the category of noncommutative DG correspondences. The category $\DGcorr$ is equivalent to $\Hmo_0$. Moreover, in Theorem 3.7 of \cite{KQ} actually the connective version of Theorem \ref{NCK} was proven. The translation to the nonconnective version based on the above discusion (see subsection \ref{NegK}) is straightforward. Note that the functor $A\mapsto \pi_*(\bK(\hpf(A\prot\cpt)))$ is $C^*$-stable, half-exact, and homotopy invariant whence it is Bott $2$-periodic and the natural {\em comparison map} (see Section \ref{AlgVsTop}) gives rise to a natural transformation $\pi_*(\bK(\hpf(A\prot\cpt)))\map\KT_*(A)$ between two $2$-periodic homology theories on $\Csep$.  
\end{rem}

\noindent
A crucial insight of Rosenberg in \cite{RosCT} is that certain bundles of compact operators $\cpt$ on locally compact spaces can be used to model twisted $\K$-theory that was introduced in \cite{DonKar}. More precisely, given any pair $(E,h)$ with $E$ locally compact one can construct a noncommutative stable $C^*$-algebra $\CT(E,h)$, whose topological $\K$-theory is the twisted $\K$-theory of the pair $(E,h)$. This formalism extends to certain infinite dimensional spaces through the use of $\sC$-algebras \cite{MyTwist}. In \cite{BMRS,BMRS2} the authors extended the formalism of $\TT$-duality to $C^*$-algebras and showed that under favourable circumstances if $B$ and $B'$ are $\TT$-dual $C^*$-algebras, then there is an invertible element in $\KK_0(B,\Sigma B')$ that implements the twisted $\K$-theory isomorphism (as in \eqref{Twist}). The Connes--Skandalis picture of $\KK$-theory \cite{ConSka} and Rieffel's imprimitivity result \cite{IndRepRie} are pertinent to their construction. Thanks to Theorem \ref{Tdual} we conclude that if two stable $C^*$-algebras $B$ and $B'$ are $\TT$-dual, such that there is an invertible element $\alpha\in\KK_0(B,\Sigma B')$, then their noncommutative motives $\hpf(B)$ and $\hpf(B')$ are isomorphic in $\Hmo_0$. Furthermore, Theorem \ref{NCK} asserts that the invertible element $\alpha$ implements the twisted $\K$-theory isomorphism (like \eqref{Twist}) that is expected from $\TT$-duality.

Recall that the Cuntz algebra $\Oinf$ is the universal unital $C^*$-algebra generated by a set of isometries $\{s_i\,|\, i\in\NN\}$ with mutually orthogonal range projections $s_is_i^*$ \cite{CunO}. Observe that $\Oinf$ is a unital $C^*$-algebra, so that $\Oinf$-stabilization preserves unitality (unlike $\cpt$-stabilization). The following Lemma is crucial and it exploits the fact that $\Oinf$ is {\em purely infinite}.

\begin{lem} \label{Basic}
There is a commutative diagram in $\CAlg$ \beq \label{BasicCD}
\xymatrix{
\Oinf\ar[rr]^\iota\ar[rd]_\theta && \Oinf \\
&\Oinf\prot\cpt \ar[ur]_\kappa,
}
\eeq where the top horizontal arrow $\iota:\Oinf\map \Oinf$ is an inner endomorphism.
\end{lem}

\begin{proof}
Observe that the subset $\{s_is^*_j\,|\, i,j\in\NN\}\subset\Oinf$ generates a copy of the compact operators $\cpt$ inside $\Oinf$. Consider the $*$-homomorphism $\kappa:\Oinf\prot\cpt\map \Oinf$, which is defined as $a\otimes e_{ij}\mapsto s_i a s^*_j$. Due to the simplicity of all the $C^*$-algebras in sight, $\kappa$ is injective. Let $\theta:\Oinf\map \Oinf\prot\cpt$ be simply the corner embedding, sending $a\mapsto a\otimes e_{11}$. The composite $\iota = \kappa\theta$ is given by $\iota(a)= s_1 a s^*_1$. This $*$-homomorphism is manifestly inner. 
\end{proof}

\noindent
Recall that a functor $F:\Csep\functor\Hmo_0$ is called {\em split exact} if it sends a split exact sequence in $\Csep$ to a direct sum diagram in the additive category $\Hmo_0$. It follows from Lemma 3.1 of \cite{KQ} that the functor $\hpf(-)$ is split exact.

\begin{thm} \label{OMotive}
If $A$ and $A'$ are isomorphic in $\KKcat$, then the noncommutative motives of $A\prot\Oinf$ and $A'\prot\Oinf$ are isomorphic in $\Hmo_0$.
\end{thm}

\begin{proof}
 Let us first assume that $A,A'$ are unital and let $\alpha\in\KK_0(A,A')$ be any invertible element. Consider the commutative diagram that is obtained by applying $A\prot-$ to the commutative diagram \ref{BasicCD} \beq \label{RS}
\xymatrix{
A\prot\Oinf\ar[rr]^{\id_A\prot\iota}\ar[rd]_{R:=\id_A\prot \theta} && A\prot\Oinf\\
& A\prot\Oinf\prot\cpt.\ar[ru]_{S:=\id_A\prot \kappa}}
\eeq Now from Theorem \ref{Tdual} one obtains a diagram in $\Hmo_0$ \beq \label{Motive}
\xymatrix{
\hpf(A\prot\Oinf)\ar[r]^{\hpf(R)} & \hpf(A\prot\Oinf\prot\cpt) \ar[r]^{\hpf(S)}\ar[d]^{\beta=\hpf(\alpha\prot\id_{\Oinf}\prot\id_\cpt)} & \hpf(A\prot\Oinf) \\
\hpf(A'\prot\Oinf)\ar[r]^{\hpf(R')} & \hpf(A'\prot\Oinf\prot\cpt) \ar[r]^{\hpf(S')} & \hpf (A'\prot\Oinf). 
}
\eeq where $R'$ and $S'$ are defined in the obvious manner (replace $A$ by $A'$ in diagram \ref{RS}). Since $\alpha$ is invertible, so are $\alpha\prot\id_{\Oinf}$ and $\alpha\prot\id_{\Oinf}\prot\id_{\cpt}$. Therefore, the middle vertical arrow $\beta$ is an isomorphism. Observe that $S\circ R$ is an inner endomorphism in $\Csep$ of the form $x\mapsto (\one_A\otimes s_1)x(\one_A\otimes s_1)^*$ (and so is $S'\circ R'$ similarly). It is known that if $F$ is a matrix stable functor on $\CAlg$ (resp. $\Csep$) and $f$ is an inner endomorphism in $\CAlg$ (resp. $\Csep$), then $F(f)$ is the identity map (see, for instance, Proposition 3.16. of \cite{CunMeyRos}). It was shown in Lemma 2.3 of \cite{KQ} that the functor $\hpf(-)$ is matrix stable on $\Csep$, whence we get $$\hpf(S)\circ\hpf(R)=\id_{\hpf(A\prot\Oinf)} \quad\text{ and }\quad \hpf(S')\circ\hpf(R')=\id_{\hpf(A'\prot\Oinf)}.$$

\noindent
Thus the maps $\hpf(R)$ and $\hpf(R')$ possess left inverses. An inspection of diagram \ref{Motive} reveals that it suffices to show that they also possess right inverses. The composite $*$-homomorphism $\cpt\overset{i}{\hookrightarrow}\Oinf\overset{\theta}{\map}\Oinf\prot\cpt$ defines an invertible element $\theta\circ i=\gamma\in\KK_0(\cpt,\Oinf\prot\cpt)$. Consequently, $\id_A\prot\gamma\in\KK_0(A\prot\cpt, A\prot\Oinf\prot\cpt)$ is an invertible element. By Theorem \ref{Tdual} $\id_A\prot\gamma\prot\id_\cpt=(\id_A\prot\theta\prot\id_\cpt)\circ(\id_A\prot i\prot\id_\cpt)$ induces an isomorphism $$\hpf(A\prot\cpt\prot\cpt)\overset{\sim}{\map}\hpf(A\prot\Oinf\prot\cpt\prot\cpt).$$ Let us set $I=\id_A\prot i$, so that $\hpf(R\prot\id_\cpt)\circ\hpf(I\prot\id_\cpt)$ is the above isomorphism. Now consider the following commutative diagram \beqn
\xymatrix{
A\prot\cpt \ar[r] \ar[d]^I & A\prot\cpt\prot\cpt \ar[d]^{I\prot\id_\cpt}\\
A\prot\Oinf \ar[r] \ar[d]^R & A\prot\Oinf\prot\cpt \ar[d]^{R\prot\id_\cpt}\\
A\prot\Oinf\prot\cpt \ar[r] & A\prot\Oinf\prot\cpt\prot\cpt.
}
\eeqn Here all the horizontal arrows are corner embeddings. Now the top and the bottom horizontal arrows are homotopic to isomorphisms. Since $\hpf(-)$ is homotopy invariant on stable $C^*$-algebras, it sends the top and the bottom horizontal arrows to isomorphisms. We already know that it sends $(R\prot\id_\cpt)\circ (I\prot\id_\cpt)$ to an isomorphism. It follows that $\hpf(R)$ has a right inverse. Similarly, one can prove that $\hpf(R')$ has a right inverse. Now using split exactness of $\hpf(-)$ one can extend the result to nonunital $C^*$-algebras.
\end{proof}

\begin{cor} \label{KKfac}
 The functor $\hpf(-\prot\Oinf)$ is $C^*$-stable and it factors through $\KKcat$.
\end{cor}

\begin{proof}
 For any separable $C^*$-algebra $A$ the corner embedding $A\map A\prot\cpt$ is $\KKcat$-invertible whence $\hpf(-\prot\Oinf)$ is $C^*$-stable. It follows from Lemma 3.1 of \cite{KQ} that the functor $\hpf(-\prot\Oinf)$ is split exact. The second assertion now is a consequence of the universal characterization of $\KKcat$.
\end{proof}

\noindent
Now we prove the $\Oinf$-analogue of Theorem \ref{NCK}.

\begin{thm} \label{OK}
 For any $A\in\Csep$ the homotopy groups of the nonconnective $\K$-theory spectrum of $\hpf(A\prot\Oinf)$ are naturally isomorphic to the topological $\K$-theory groups of $A$.
\end{thm}

\begin{proof}
 By the above Corollary the nonconnective $\K$-theory spectra of $\hpf(A\prot\Oinf)$ and $\hpf(A\prot\Oinf\prot\cpt)$ are weakly equivalent. By Theorem \ref{NCK} the homotopy groups of the nonconnective $\K$-theory spectrum of $\hpf(A\prot\Oinf\prot\cpt)$ are isomorphic to the topological $\K$-theory groups of $A\prot\Oinf$, which are in turn isomorphic to those of $A$.
\end{proof}

\begin{rem}[Categorification of topological $\TT$-duality]
It is shown in Example 5.6 of \cite{BMRS2} that if two pairs $(E_1,h_1)$ and $(E_2,h_2)$ over $B$ with are $\TT$-dual, then one can construct an invertible element in $\KK_1(\CT(E_1,h_1),\CT(E_2,h_2))$. It follows that $\CT(E_1,h_1)$ and $\Sigma\CT(E_2,h_2)$ are isomorphic in $\KKcat$ whence by Theorem \ref{OMotive} $\hpf(\CT(E_1,h_1)\prot\Oinf)$ and $\hpf(\Sigma\CT(E_2,h_2)\prot\Oinf)$ are isomorphic in $\Hmo_0$. Now using Theorem \ref{OK} one concludes $$\K_*(\hpf(\CT(E_i,h_i)\prot\Oinf))\cong \KT_*(\CT(E_i,h_i))\cong(\KT)^*(E_i,h_i)$$ for $i=1,2$ and hence one obtains the isomorphism \eqref{Twist} between twisted $\K$-theories (up to a shift) under $\TT$-duality. Since noncommutative motives constitute the universal additive invariant \cite{Tab3}, an isomorphism therein is the most fundamental (co)homological isomorphism. 
\end{rem}

\begin{rem}
 The introduction of $\Oinf$ in this setting is quite interesting because of the Dixmier--Douady theory via $\Oinf\prot\cpt$-bundles due to Dadarlat--Pennig \cite{DadPen} that is able to see {\em higher twists} of $\KT$-theory. This avenue of research deserves further attention.
\end{rem}

\section{The generalized homology theory $\bK(-\prot\Oinf)$} \label{AlgVsTop}
 Let $\hSp$ denote the triangulated stable homotopy category. A functor $F:\CAlg\functor\hSp$ is called {\em homotopy invariant} if it sends the evaluation at $t$ map $\ev_t:A[0,1]\map A$ to an isomorphism in $\hSp$ for all $A\in\CAlg$. Such a functor is called {\em excisive} if for any short exact sequence $0\map A\map B\map C\map 0$ in $\CAlg$ the induced diagram $F(A)\map F(B)\map F(C)\map \Sigma F(A)$ is an exact triangle in $\hSp$. A homotopy invariant excisive functor $F:\CAlg\functor\hSp$ is called an $\hSp$-valued {\em generalized homology theory} on $\CAlg$. It is known that the algebraic $\K$-theory functor $\bK(-)$ acquires special properties after stabilization with respect to the compact operators. We are going to show that the same is true after $\Oinf$-stabilization. 

\begin{prop} \label{excision}
The functor $\bK(-\prot\Oinf):\CAlg\functor\hSp$ is an excisive functor.
\end{prop}

\begin{proof}
It follows from the Suslin--Wodzicki Theorem \cite{SusWod1,SusWod2} that the functor $\bK$ is excisive on $\CAlg$. Since maximal $C^*$-tensor product is exact, the functor $-\prot\Oinf$ preserves exactness in $\CAlg$ whence $\bK(-\prot\Oinf)$ is excisive.
\end{proof}

Thanks to the Karoubi conjecture, which is now a Theorem \cite{SusWod1,SusWod2}, we know that the nonconnective algebraic $\K$-theory of a stable $C^*$-algebra is isomorphic to its topological $\K$-theory. In fact, there is a canonical comparison map of spectra that induces the isomorphisms $c_n(A):\K_n(A)\map\KT_n(A)$ for all $n\in\ZZ$ when $A$ is stable \cite{KaroubiConj} (see also \cite{RosComparison}). The comparison map $c_0(A):\K_0(A)\map\KT_0(A)$ is always an isomorphism.

\begin{prop}[Corti{\~n}as--Phillips, Karoubi--Wodzicki] \label{CorPhi}
For any $C^*$-algebra $A$ the comparison map $c_n(A\prot\Oinf):\K_n(A\prot\Oinf)\map\KT_n(A\prot\Oinf)$ is an isomorphism for all $n\in\ZZ$.
\end{prop}

\begin{proof}
Let us first assume that $A$ is a unital $C^*$-algebra. After applying $A\prot-$ to the commutative diagram \ref{BasicCD} in Lemma \ref{Basic} we obtain \beq
\xymatrix{
A\prot\Oinf\ar[rr]^{\id\prot\iota}\ar[rd]_{R:=\id\prot \theta} && A\prot\Oinf\\
& A\prot\Oinf\prot\cpt, \ar[ru]_{S:=\id\prot \kappa}}
\eeq where the top horizontal arrow is an inner endomorphism. Now applying the functors $\K_n(-)$, $\KT_n(-)$ and using the naturality of $c_n$, we get a commutative diagram \beq
\xymatrix{\K_n(A\prot\Oinf)\ar[rr]^{\K_n(R)}\ar[d]^{c_n(A\prot\Oinf)} && \K_n(A\prot\Oinf\prot\cpt)\ar[rr]^{\K_n(S)}\ar[d]^{c_n(A\prot\Oinf\prot\cpt)} && \K_n(A\prot\Oinf)\ar[d]^{c_n(A\prot\Oinf)} \\
\KT_n(A\prot\Oinf)\ar[rr]^{\KT_n(R)} && \KT_n(A\prot\Oinf\prot\cpt)\ar[rr]^{\KT_n(S)} && \KT_n(A\prot\Oinf).
}
\eeq Since $S\circ R$ is the inner endomorphism $\id\prot\iota: A\prot\Oinf\map A\prot\Oinf$, we conclude that $\K_n(S)\circ\K_n(R)$ is the identity map due to the matrix stability of algebraic $\K$-theory on the category of unital $C^*$-algebras. Moreover, $\KT_n(S)\circ\KT_n(R)$ is also the identity map due to the matrix stability of $\KT_n(-)$. The assertion for unital $A$ now follows by a simple diagram chase. Indeed, it is easily seen that $\K_n(R)$ must be injective and $\KT_n(S)$ must be surjective. Since $A\prot\Oinf\prot\cpt$ is stable, we conclude that $c_n(A\prot\Oinf\prot\cpt)$ is an isomorphism. Thus $c_n(A\prot\Oinf\prot\cpt)\circ\K_n(R)$ is injective whence so is $c_n(A\prot\Oinf)$ (the left vertical one). Similarly, $\KT_n(S)\circ c_n(A\prot\Oinf\prot\cpt)$ is surjective whence so is $c_n(A\prot\Oinf)$ (the right vertical one).

\noindent
The proof for nonunital $A$ follows by a simple excision argument (see Proposition \ref{excision}). 
\end{proof}

\begin{rem}
The above Proposition also follows from a result of Corti{\~n}as--Phillips. They proved that the comparison map $c_n(A):\K_n(A)\map\KT_n(A)$ is an isomorphism in a more general setting \cite{CorPhi}. Our argument above is based on a strategy of Karoubi--Wodzicki \cite{KarWod} with some simplifications that exploit the special properties of $\Oinf$. We have decided to include our simple proof as it involves only elementary (homological) algebra and hence it is (hopefully) comprehensible to non-experts on $C^*$-algebras. 
\end{rem}

Recall from \cite{JoaKHom} (see also Section 8.3 of \cite{CunMeyRos}) that for any $C^*$-algebra $A$ there is a functorial spectrum $\bK^\textup{top}(A)$, whose homotopy groups are the topological $\K$-theory groups of $A$.

\begin{thm} \label{GenHom}
For every $A\in\CAlg$ there is a natural isomorphism $\bK(A\prot\Oinf)\cong\bK^{\textup{top}}(A)$ in $\hSp$, i.e., the functor $\bK(-\prot\Oinf):\CAlg\functor\hSp$ is a model for topological $\K$-theory.
\end{thm}

\begin{proof}
It follows from Proposition \ref{CorPhi} that the natural comparison map of spectra is a weak equivalence. Since the canonical $*$-homomorphism $A\map A\prot\Oinf$ sending $a\mapsto a\otimes \one_{\Oinf}$ is a $\KK$-equivalence, we have a zigzag of weak equivalences of spectra $$\bK(A\prot\Oinf)\overset{\sim}{\map}\bK^\textup{top}(A\prot\Oinf)\overset{\sim}{\leftarrow}\bK^\textup{top}(A).$$ Thus for every $A\in\CAlg$ there is a natural isomorphism $\bK(A\prot\Oinf)\cong\bK^{\textup{top}}(A)$ in $\h\Sp$.
\end{proof}

\begin{rem} \label{HSS}
It is interesting to associate with any $C^*$-algebra a symmetric spectrum, whose homotopy groups are the topological $\K$-theory groups. One strategy that circumvents the intricacies involved in constructing symmetric spectra of topological $\K$-theory directly is the following: for any unital $C^*$-algebra $A$ construct the (connective) Waldhausen $\K$-theory \cite{Waldhausen} of $A\prot\Oinf$; it produces naturally a symmetric spectrum \cite{GeiHes} and thanks to Theorem \ref{GenHom} it is a model for the (connective) topological $\K$-theory of $A$. The Green--Julg--Rosenberg Theorem also enables us to treat the $G$-equivariant case for a finite group $G$. Indeed, one can simply apply the (connective) Waldhausen $\K$-theory functor to $(A\rtimes G)\prot \Oinf$ for any unital $G$-$C^*$-algebra $A$ and use the natural identifications $$\K_*((A\rtimes G)\prot\Oinf)\cong\KT_*((A\rtimes G)\prot\Oinf)\cong\KT_*(A\rtimes G)\cong(\KT)^G_*(A).$$ Using Theorem \ref{OK} we can handle the situation if $A$ is nonunital or $G$ is not finite (but compact) or both. This construction would radically differ from that of \cite{JoaKHom}.
\end{rem}

\section{Strongly self-absorbing operads} \label{SSOperad}
Operadic structures have pervaded many areas of mathematics and physics with wide ranging applications. From the viewpoint of topology the operadic machinery can be effectively used to recognise (infinite) loop spaces. An {\em operad} in a symmetric monoidal category $(\cC,\otimes,\one_\cC)$ consists of a collection of objects $\{C(j)\}_{j\geqslant 0}$ with each $C(j)$ carrying a right action of the permutation group $\Sigma_j$, a unit map $\eta:\one_\cC\map C(1)$, and product or composition maps $$\gamma=\gamma_{j_1,\cdots,j_k}: C(k)\otimes C(j_1)\otimes\cdots\otimes C(j_k)\map C(j)$$ for $k\geqslant 1$ and $j_s\geqslant 0$ for all $s=1,\cdots, k$ subject to $j=\sum_{s} j_s$. These data should be intercompatible in a specific manner, i.e., satisfy certain associativity, unitality, and equivariance axioms (see, for instance, \cite{KriMay}). Neglecting the actions of the permutation groups and the corresponding equivariance conditions one arrives at the notion of a {\em nonsymmetric operad}. In this section a space is tacitly assumed to be compactly generated and weakly Hausdorff. Observe that such spaces constitute a symmetric monoidal category under cartesian product with $\pt$ (a singleton space) as a unit object and hence one may consider operads in spaces.

Recall that a unital separable $C^*$-algebra $\cD$ ($\cD\neq\CC$) is called {\em strongly self-absorbing} if there is an isomorphism $\cD\overset{\sim}{\map}\cD\prot\cD$ that is approximately unitarily equivalent to the first factor embedding $\cD\map\cD\prot\cD$ sending $d\mapsto d\otimes \one_\cD$ \cite{TomWin}. Such $C^*$-algebras turn out to be simple and nuclear. The Cuntz algebra $\Oinf$ is a prominent example of such a $C^*$-algebra. For any strongly self-absorbing $C^*$-algebra $\cD$ we set $\cD(j) = \Hom_1(\cD^{\prot j},\cD)$, i.e., the space of unital full $*$-homomorphisms $\cD^{\prot j}\map\cD$ with the point-norm topology. Since $\cD$ is a separable $C^*$-algebra, it follows from Lemma 22 of \cite{Mey1} that each $\cD(j)$ is a metrizable topological space. Hence they are all compactly generated and Hausdorff spaces. 

\begin{lem} \label{OperadProm}
The collection $\{\cD(j)\}_{j\geqslant 0}$ can be promoted to an operad in spaces.
\end{lem}

\begin{proof}
Let us define $\gamma$ and $\eta$ as follows: \beqn
\gamma:\cD(k)\times \cD(j_1)\times\cdots\times \cD(j_k)&\map & \cD(j)=\cD(j_1 +\cdots + j_k)\\
(\alpha, \beta_1, \cdots, \beta_k) &\mapsto & \alpha\circ (\beta_1\otimes\cdots\otimes\beta_k)
\eeqn and $\eta: \pt \map \cD(1)$ sends the unique element in $\pt$ to $\id:\cD\map \cD$. If we let the permutation group $\Sigma_j$ act on $\cD(j)=\Hom_1(\cD^{\prot j},\cD)$ by permuting the tensor factors of $\cD^{\prot j}$, then it can be verified that the data satisfy the associativity, unitality, and equivariance axioms. 
\end{proof}

\noindent
Thanks to the above Lemma we introduce the following operad:

\begin{defn} \label{D}
 For any strongly self-absorbing $C^*$-algebra $\cD$ we call the operad that the collection $\{\cD(j)\}_{j\geqslant 0}$ defines as the {\em strongly self-absorbing $\cD$-operad}. 
 \end{defn}
 
 \begin{rem}
 For $j=0$ we get $\cD(0)=\pt$, i.e., a singleton set containing the unique unital inclusion $\CC\hookrightarrow\cD$. Hence the strongly self-absorbing $\cD$-operad is a {\em reduced operad}.
\end{rem}

\begin{prop} \label{Einfinity}
Every strongly self-absorbing $\cD$-operad is an $\E_\infty$-operad.
\end{prop}

\begin{proof}
 We need to show that each $\cD(j)$ for $j\geqslant 0$ is contractible and the action of $\Sigma_j$ on each $\cD(j)$ is free. The contractibility of each $\cD(j)$ follows from Theorem 2.3 of \cite{DadPen} and the fact that $\cD\cong\cD^{\prot j}$. In order to see the freeness of the $\Sigma_j$-action on $\cD(j)$ we check the stabilizers. For any $f\in\cD(j)$ suppose $f\sigma = f$. Owing to the simplicity of $\cD^{\prot n}$ any such $f\in \cD(j)$ must be a monomorphism whence $\sigma$ has to be the trivial permutation. 
\end{proof}

\begin{rem}
The underlying nonsymmetric operad of every strongly self-absorbing $\cD$-operad is an $\A_\infty$-operad.
\end{rem}

A functor $G:(\cE,\otimes_\cE,\one_\cE)\functor (\cF,\otimes_\cF,\one_\cF)$ between symmetric monoidal categories is called {\em lax symmetric monoidal} if there is a morphism $\one_\cF\map G(\one_\cE)$ in $\cF$ and natural transformations \beq \label{monoidal} \kappa: G(A)\otimes_\cF G(B)\map G(A\otimes_\cE B)\eeq for all $A,B\in\cE$ that satisfy certain well-known associativity, unitality, and symmetry conditions. If $\cE$ and $\cF$ are (pointed) topological categories, i.e., they are enriched over (pointed) spaces, then a functor $G:\cE\map\cF$ is called {\em enriched} if for all $A,B\in\cE$ the induced map $\Map_\cE(A,B)\map\Map_\cF(G(A),G(B))$ is (pointed) continuous. Here we have adopted the convention that in the enriched setting we denote the (pointed) space of morphisms by $\Map_\cE(-,-)$, $\Map_\cF(-,-)$, and so on. A symmetric monoidal pointed topological category $\cB=(\cB,\wedge,\one_\cB)$ is said to be equipped with a closed action of pointed spaces $\cS_*$ if, for every $X\in\cB$ and for every pair of pointed spaces $K,L$, the following hold:

\begin{itemize}
 \item the functor $(-)\wedge X: \cS_*\functor\cB$ is the (enriched) left adjoint of $\Map_\cB(X,-):\cB\functor\cS_*$, 
  \item the functor $K\wedge (-): \cB\functor\cB$ admits an (enriched) right adjoint $(-)^K:\cB\functor\cB$, and
 \item there are coherent natural isomorphisms $(K\wedge L)\wedge X\cong K\wedge (L\wedge X)$ and $S^0\wedge X\cong X$.
\end{itemize} Let $C$ be any operad in spaces. An object $X\in\cB$ is said to be an {\em algebra over $C$} if there are maps $\theta: C(j)_+\wedge X^{\wedge j}\map X$ in $\cB$ for all $j\geqslant 0$, that are associative, unital, and equivariant in a suitable sense \cite{KriMay}. Moreover, an object $M\in\cB$ is said to be a {\em module over $X$} if there are maps $\lambda: C(j)_+\wedge X^{\wedge (j-1)}\wedge M\map M$ in $\cB$ for all $j\geqslant 1$ that are again associative, unital, and equivariant as explained in \cite{KriMay}. One can also define $X$ as an algebra object in $\cB$ over $C$ via a morphism of topological operads $C_+\map\mathtt{End}_\cB(X)$, where $\mathtt{End}_\cB(X)$ denotes the endomorphism operad of $X$ in $\cB$ (so that $\mathtt{End}_\cB(X)(j)=\Map_\cB(X^{\wedge j},X)$).

A typical example for $\cB$ that the reader should keep in mind is the category of {\em symmetric spectra} $\Sps$. It has an associative and commutative smash product $\wedge$ (with the sphere spectrum $\mathbb{S}$ as the unit object) that is accompanied by a well developed theory of rings and modules. The category of symmetric spectra in pointed spaces is enriched, tensored and cotensored over $\cS_*$, i.e., it admits a closed action of $\cS_*$ (see Propositions 1.3.1 and 1.3.2 of \cite{HSS} and Example 3.36 of \cite{SchwedeBook} for their topological counterparts). Therefore, the symmetric monoidal category $(\Sps,\wedge,\mathbb{S})$ satisfies the assumptions on $(\cB,\wedge,\one_\cB)$ mentioned above. It is also a stable model category, whose homotopy category is equivalent to $\hSp$. For further details consult \cite{HSS,SchwedeBook}.

Recall that $(\CAlg,\prot,\CC)$ is a symmetric monoidal category, where $\prot$ is the maximal (or the minimal) $C^*$-tensor product. It is also enriched over (pointed) spaces if we endow the morphism sets with the point-norm topology. Let $\cC$ denote a symmetric monoidal topological subcategory of that of $C^*$-algebras that contains all strongly self-absorbing $C^*$-algebras, e.g., $\cC=\Csep$ or the category of nuclear separable $C^*$-algebras. In each case $\prot$ can be either the maximal or the minimal $C^*$-tensor product.

\begin{thm} \label{Operad}
 Let $\cD$ be any strongly self-absorbing $C^*$-algebra and $(\cB,\wedge,\one_\cB)$ be any symmetric monoidal pointed topological category that is equipped with a closed action of pointed spaces. Let $\cC$ be a symmetric monoidal topological subcategory of $\CAlg$ as above and $F:\cC\functor\cB$ be a lax symmetric monoidal enriched functor. Then $F(\cD)$ is an algebra object in $\cB$ over the strongly self-absorbing $\cD$-operad. Moreover, for any $A\in\cC$, the object $F(\cD\prot A)\cong F(A\prot \cD)$ is a module over $F(\cD)$. 
\end{thm}

\begin{proof}
 We define for all $j\geqslant 0$ the following maps \beqn \theta: \cD(j)_+\wedge F(\cD)^{\wedge j} &\map& F(\cD) \\ (f,(x_1,\cdots,x_j)) &\mapsto& f_*(\kappa(x_1,\cdots,x_j)). \eeqn Here $f_*:F(\cD^{\prot j})\map F(\cD)$ is the map induced by $f\in\cD(j)$ and $\kappa:F(\cD)^{\wedge j}\map F(\cD^{\prot j})$ is the canonical map induced by \eqref{monoidal}. Similarly, we define for all $j\geqslant 1$ the following maps \beqn \lambda: \cD(j)_+\wedge F(\cD)^{\wedge (j-1)}\wedge F(\cD\prot A) &\map&  F(\cD\prot A) \\ (f,(x_1,\cdots,x_{j-1}),y) &\mapsto& (f\otimes\id)_*(\kappa'((x_1,\cdots,x_j),y)). \eeqn Here $\kappa': F(\cD)^{\wedge (j-1)}\wedge F(\cD\prot A)\map F(\cD^{\prot (j-1)})\wedge F(\cD\prot A)\map F(\cD^{\prot j}\prot A)$ is induced by the composition of the canonical maps furnished by \eqref{monoidal} and $(f\otimes\id)_*: F(\cD^{\prot j}\prot A)\map F(\cD\prot A)$ is the map induced by $f\otimes\id: \cD^{\prot j}\prot A\map \cD\prot A$. 
 
 We claim that $\theta$ and $\lambda$ are morphisms in $\cB$. The enriched functor $F$ gives a continuous map $\cD(j)_+\map \Map_\cB(F(\cD^{\prot j}),F(\cD))$. There is also a continuous map $\Map_\cB(F(\cD^{\prot j}),F(\cD))\map \Map_\cB(F(\cD)^{\wedge j},F(\cD))$ induced by $\kappa:F(\cD)^{\wedge j}\map F(\cD^{\prot j})$. The composite continuous map $\cD(j)_+\map \Map_\cB(F(\cD)^{\wedge j},F(\cD))$ translates to a map $\cD(j)_+\wedge F(\cD)^{\wedge j} \map F(\cD)$ in $\cB$ via the closed action of $\cS_*$ on $\cB$, which is seen to be $\theta$. Similar arguments show that $\lambda$ is also a map in $\cB$. The axiom for unitality says that the following diagrams commute:
 
\beqn
\xymatrix{
S^0\wedge  F(\cD)\ar[d]_{\eta\wedge\id}\ar[r]^\cong &  F(\cD) &&  S^0\wedge  F(\cD\prot A)\ar[d]_{\eta\wedge\id}\ar[r]^\cong &  F(\cD\prot A)\\
\cD(1)_+\wedge F(\cD)\ar[ru]_\theta &&& \cD(1)_+\wedge F(\cD\prot A).\ar[ru]_\lambda
}
\eeqn This condition is clear from the fact that $\eta$ maps the non-basepoint in $S^0$ to $\id:\cD\map\cD$ (see Lemma \ref{OperadProm}). Now using the hypothesis that $F(-)$ is a lax symmetric monoidal functor one can check the required associativity and equivariance conditions. 
\end{proof}

\begin{ex} \label{OpAction}
We demonstrate the utility of our result with an important and interesting example. There is a construction of the topological $\K$-theory spectrum $\Kw(-)$ of a separable $C^*$-algebra with values in symmetric spectra $\Sps$ that satisfies the hypotheses of the above Theorem (see Lemma 3.4 of \cite{DelEmeKanMey}). Note that the authors work in the symmetric monoidal category of $\ZZ/2$-graded separable $C^*$-algebras $\Csep^{\ZZ/2}$ and use the graded minimal $C^*$-tensor product for the symmetric monoidal structure thereon. Hence we equip $\Csep$ with the minimal $C^*$-tensor product and consider the (lax) symmetric monoidal enriched functor $\Csep\functor\Csep^{\ZZ/2}$ that endows any $A\in\Csep$ with the trivial $\ZZ/2$-grading. The functor $\Kw:\Csep^{\ZZ/2}\functor\Sps$ works even for separable $C^*$-algebras equipped with an additional compact group action. It is a variation of the construction in \cite{BunJoaSto} that landed in orthogonal spectra. The enrichment of the composite functor $\Csep\functor\Csep^{\ZZ/2}\overset{\Kw}{\functor}\Sps$ follows from the fact that $\Csep$ and $\Csep^{\ZZ/2}$ are themselves enriched over (pointed) spaces and $\Kw(-)$ is defined as a symmetric sequence of mapping spaces in $\Csep^{\ZZ/2}$ (see Section 3.1 of \cite{DelEmeKanMey}).
\end{ex}

\begin{rem}
Thanks to Proposition \ref{Einfinity} and Theorem \ref{Operad}, for any strongly self-absorbing $C^*$-algebra $\cD$, one might call $\Kw(\cD)$ an $\E_\infty$-algebra object in symmetric spectra. Using Remark 0.14 of \cite{MMSS} (see also Theorem 1.4 of \cite{ElmMan}) one can rectify this $\E_\infty$-algebra structure on $\Kw(\cD)$ (resp. the module structure on $\Kw(\cD\prot A)$) to a strictly commutative algebra structure in $\Sps$ (resp. to a strict module structure over the latter).
\end{rem}

\section{$\K$-regularity of $\Oinf$-stable $C^*$-algebras}
Let $F$ be any functor on $\CAlg$. A $C^*$-algebra $A$ is called {\em $F$-regular} if the canonical inclusion $A\map A[t_1,\cdots,t_n]$ induces an isomorphism $F(A)\overset{\sim}{\map} F(A[t_1,\cdots ,t_n])$ for all $n\in\NN$. This map has a one-sided inverse induced by the evaluation map $\ev_0$. Rosenberg conjectured that any $C^*$-algebra $A$ is $\K_0$-regular. Using the techniques developed to prove the Karoubi conjectures  \cite{Hig2}, it is shown in Theorem 3.4 of \cite{RosComparison} that the conjecture is true if $A$ is stable. In fact, the Theorem in \cite{RosComparison} asserts that a stable $C^*$-algebra is $\K_m$-regular for all $m\in\ZZ$. A $C^*$-algebra is called {\em $\K$-regular} if it is $\K_m$-regular for all $m\in\ZZ$.

\begin{thm} \label{Kreg}
The $C^*$-algebras $A\prot\Oinf$ are $\K$-regular for all $A\in\CAlg$.
\end{thm}

\begin{proof}
In order to avoid notational clutter let us set $B[n]:= B[t_1,\cdots,t_n]$ for any $B\in\CAlg$. Using excision we may assume that $A$ is unital. Arguing as in the proof of Proposition \ref{CorPhi} we obtain a commutative diagram

\beq
\xymatrix{
\K_m(A\prot\Oinf)\ar[r]\ar[d] & \K_m(A\prot\Oinf\prot\cpt) \ar[r]\ar[d] & \K_m(A\prot\Oinf)\ar[d] \\
\K_m((A\prot\Oinf)[n])\ar[r] & \K_m((A\prot\Oinf\prot\cpt)[n]) \ar[r] & \K_m ((A\prot\Oinf)[n]). }
\eeq Due to the stability of $A\prot\Oinf\prot\cpt$ the middle vertical arrow is an isomorphism. Moreover, the compositions of the top and the bottom horizontal arrows are again isomorphisms due to the matrix stability of the functor $\K_m(-)$ for unital algebras. Observe that the composite $*$-homomorphisms $A\prot\Oinf\map A\prot\Oinf\prot\cpt\map A\prot\Oinf$ and $(A\prot\Oinf)[n]\map (A\prot\Oinf\prot\cpt)[n]\map (A\prot\Oinf)[n]$ are still inner. Now a similar diagram chase as before enables one to conclude that the left vertical arrow must be an isomorphism.
\end{proof}

\begin{rem}
 Purely infinite simple $C^*$-algebras like $\Oinf$ can be regarded as {\em maximally noncommutative}. Rather surprisingly, one needs fairly sophisticated techniques to establish the $\K$-regularity of commutative $C^*$-algebras (see \cite{RosComparison,CorTho2}). 
\end{rem}

\begin{rem}
 Using Proposition \ref{CorPhi} and Theorem \ref{Kreg} the {\em reduction principle} for assembly maps (see Theorem 1.1 of \cite{MyNovikov}) can be generalized to include $\Oinf$ as coefficients, i.e., for a countable, discrete, and torsion free group $G$, if the Baum--Connes assembly map with complex coefficients is injective (resp. split injective), then the Farrell--Jones assembly map in algebraic $\K$-theory with $\Oinf$-coefficients is also injective (resp. split injective).
\end{rem}

\section{Algebraic $\K$-theory of certain $\Oinf$-stable $C^*$-algebras}
We now explicitly compute the algebraic $\K$-theory groups of certain $\Oinf$-stable $C^*$-algebras. It must be noted that complete calculation of the algebraic $\K$-theory groups of an arbitrary ring is an extremely difficult task in general.

\subsection{Semigroup $C^*$-algebras coming from number theory}
A recent result of Li asserts that for a countable integral domain $R$ with vanishing Jacobson radical (which is, in addition, not a field) the left regular $ax+b$-semigroup $C^*$-algebra $C^*_\lambda(R\rtimes R^\times)$ is $\Oinf$-absorbing, i.e., $C^*_\lambda(R\rtimes R^\times)\prot\Oinf\cong C^*_\lambda(R\rtimes R^\times)$ (see Theorem 1.3 of \cite{LiAxb}). Now we focus on the object of our interest, namely, the left regular $ax+b$-semigroup $C^*$-algebra $C^*_\lambda(R\rtimes R^\times)$ of the ring of integers $R$ of a number field $K$. It is shown in \cite{CunEchLi1} that in this case $$\KT_*(C^*_\lambda(R\rtimes R^\times))\cong \underset{{[X]\in G\setminus\mathcal{I}}}{\oplus} \KT_*(C^*(G_X)),$$ where $\mathcal{I}$ is the set of fractional ideal of $R$, $G=K\rtimes K^\times$, and $G_X$ is the stabilizer of $X$ under the $G$-action on $\mathcal{I}$. The orbit space $G\setminus \mathcal{I}$ can be identified with the ideal class group of $K$.

\noindent
As a consequence of Proposition \ref{CorPhi} we obtain

\begin{thm} \label{ax+b}
 The algebraic $\K$-theory of the $ax+b$-semigroup $C^*$-algebra of the ring of integers $R$ of a number field $K$ is $2$-periodic and explicitly given by $$\K_*(C^*_\lambda(R\rtimes R^\times))\cong\underset{{[X]\in G\setminus\mathcal{I}}}{\oplus} \KT_*(C^*(G_X)).$$
\end{thm}

\subsection{$\Oinf$-stabilized noncommutative tori}
We recall some basic material before stating our result. A good reference for generalities on noncommutative tori is Rieffel's survey \cite{RieNCTori}. For any real-valued skew bilinear form $\theta$ on $\ZZ^n$ ($n\geqslant 2$) the $C^*$-algebra of the noncommutative $n$-torus $A^n_\theta$ can be defined as the universal $C^*$-algebra generated by unitaries $U_x\in\ZZ^n$ subject to the relation $$U_x U_y = \textup{exp}(\pi i \theta(x,y))U_{x+y} \quad\quad\quad\forall x,y\in\ZZ^n.$$ Using the Pimnser--Voiculescu exact sequence one can compute the $\KT$-theory of $A^n_\theta$ as an abelian group, namely, \beq \label{Ktop} \KT_0(A^n_\theta)\simeq \ZZ^{2^{n-1}} \quad \text{ and }\quad \KT_1(A^n_\theta)\simeq \ZZ^{2^{n-1}}.\eeq

\begin{thm} \label{NCTori}
 The algebraic $\K$-theory of the $\Oinf$-stabilized noncommutative $n$-torus $A^n_\theta$ is $2$-periodic and explicitly given by $$\K_0(A^n_\theta\prot\Oinf)\simeq \ZZ^{2^{n-1}} \quad \text{ and }\quad \K_1(A^n_\theta\prot\Oinf)\simeq \ZZ^{2^{n-1}}.$$
\end{thm}

\begin{proof}
 By Proposition \ref{CorPhi} one has an isomorphism $\K_*(A^n_\theta\prot\Oinf)\cong\KT_*(A^n_\theta\prot\Oinf)$. Observe that $\KT_0(\Oinf)\simeq\ZZ$ and $\KT_1(\Oinf)\simeq 0$ and all $C^*$-algebras in sight belong to the UCT-class. Using the K{\"u}nneth Theorem (or possibly $\Oinf$-stability in $\KT$-theory) one now deduces that $\KT_*(A^n_\theta\prot\Oinf)\cong\KT_*(A^n_\theta)$. Now use Equation \eqref{Ktop}.
\end{proof}

\noindent
We just determined the isomorphism type of the algebraic $\K$-theory groups of $A^n_\theta\prot\Oinf$. One can also describe the elements in these groups using Rieffel's results in \cite{RieProj}.

\begin{rem} \label{RieElt}
 It follows from \cite{RieProj} that for irrational $\theta$ the projections in $A^n_\theta$ generate all of $\K_0(A^n_\theta\prot\Oinf)$ and $$\K_1(A^n_\theta\prot\Oinf)\cong\KT_1(A^n_\theta\prot\Oinf)\cong\KT_1(A^n_\theta)\overset{\sim}{\leftarrow}UA^n_\theta/U^0A^n_\theta.$$ Here $UA^n_\theta$ denotes the group of unitary elements in $A^n_\theta$ and $U^0A^n_\theta$ denotes the connected component of the identity element of $UA^n_\theta$. Thus one obtains a good description of the elements of the algebraic $\K$-theory groups in low degrees in terms of projections and unitaries.
\end{rem}


\bibliographystyle{abbrv}

\bibliography{/home/ibatu/Professional/math/MasterBib/bibliography}

\begin{thebibliography}{10}

\bibitem{BluGepTab}
A.~J. Blumberg, D.~Gepner, and G.~Tabuada.
\newblock A universal characterization of higher algebraic {K}-theory.
\newblock {\em Geom. Topol.}, 17(2):733--838, 2013.

\bibitem{BEM2}
P.~Bouwknegt, J.~Evslin, and V.~Mathai.
\newblock {$T$}-duality: topology change from {$H$}-flux.
\newblock {\em Comm. Math. Phys.}, 249(2):383--415, 2004.

\bibitem{BMRS}
J.~Brodzki, V.~Mathai, J.~Rosenberg, and R.~J. Szabo.
\newblock D-branes, {RR}-fields and duality on noncommutative manifolds.
\newblock {\em Comm. Math. Phys.}, 277(3):643--706, 2008.

\bibitem{BMRS2}
J.~Brodzki, V.~Mathai, J.~Rosenberg, and R.~J. Szabo.
\newblock Non-commutative correspondences, duality and {D}-branes in bivariant
  {$K$}-theory.
\newblock {\em Adv. Theor. Math. Phys.}, 13(2):497--552, 2009.

\bibitem{BunJoaSto}
U.~Bunke, M.~Joachim, and S.~Stolz.
\newblock Classifying spaces and spectra representing the {$K$}-theory of a
  graded {$C^*$}-algebra.
\newblock In {\em High-dimensional manifold topology}, pages 80--102. World
  Sci. Publ., River Edge, NJ, 2003.

\bibitem{BunNik}
U.~Bunke and T.~Nikolaus.
\newblock T-{D}uality via {G}erby {G}eometry and {R}eductions.
\newblock {\em arXiv:1305.6050}.

\bibitem{BunRumSch}
U.~Bunke, P.~Rumpf, and T.~Schick.
\newblock The topology of {$T$}-duality for {$T^n$}-bundles.
\newblock {\em Rev. Math. Phys.}, 18(10):1103--1154, 2006.

\bibitem{BunSch1}
U.~Bunke and T.~Schick.
\newblock On the topology of {$T$}-duality.
\newblock {\em Rev. Math. Phys.}, 17(1):77--112, 2005.

\bibitem{BunSchSpiTho}
U.~Bunke, T.~Schick, M.~Spitzweck, and A.~Thom.
\newblock Duality for topological abelian group stacks and {$T$}-duality.
\newblock In {\em {$K$}-theory and noncommutative geometry}, EMS Ser. Congr.
  Rep., pages 227--347. Eur. Math. Soc., Z\"urich, 2008.

\bibitem{ConBook}
A.~Connes.
\newblock {\em Noncommutative geometry}.
\newblock Academic Press Inc., San Diego, CA, 1994.

\bibitem{ConRie}
A.~Connes and M.~A. Rieffel.
\newblock Yang-{M}ills for noncommutative two-tori.
\newblock In {\em Operator algebras and mathematical physics (Iowa City, Iowa,
  1985)}, volume~62 of {\em Contemp. Math.}, pages 237--266. Amer. Math. Soc.,
  Providence, RI, 1987.

\bibitem{ConSka}
A.~Connes and G.~Skandalis.
\newblock The longitudinal index theorem for foliations.
\newblock {\em Publ. Res. Inst. Math. Sci.}, 20(6):1139--1183, 1984.

\bibitem{CorPhi}
G.~Corti{\~n}as and N.~C. Phillips.
\newblock Algebraic {$K$}-theory and properly infinite {$C^*$}-algebras.
\newblock {\em arXiv:1402.3197}.

\bibitem{CorTho2}
G.~Corti{\~n}as and A.~Thom.
\newblock Algebraic geometry of topological spaces {I}.
\newblock {\em Acta Math.}, 209(1):83--131, 2012.

\bibitem{CunO}
J.~Cuntz.
\newblock Simple {$C\sp*$}-algebras generated by isometries.
\newblock {\em Comm. Math. Phys.}, 57(2):173--185, 1977.

\bibitem{CunKK}
J.~Cuntz.
\newblock A new look at {$KK$}-theory.
\newblock {\em $K$-Theory}, 1(1):31--51, 1987.

\bibitem{CunDenLac}
J.~Cuntz, C.~Deninger, and M.~Laca.
\newblock {$C^*$}-algebras of {T}oeplitz type associated with algebraic number
  fields.
\newblock {\em Math. Ann.}, 355(4):1383--1423, 2013.

\bibitem{CunEchLi1}
J.~Cuntz, S.~Echterhoff, and X.~Li.
\newblock On the {$K$}-theory of the {$C^*$}-algebra generated by the left
  regular representation of an {O}re semigroup.
\newblock {\em arXiv:1201.4680, to appear in J. Eur. Math. Soc.}

\bibitem{CunMeyRos}
J.~Cuntz, R.~Meyer, and J.~M. Rosenberg.
\newblock {\em Topological and bivariant {$K$}-theory}, volume~36 of {\em
  Oberwolfach Seminars}.
\newblock Birkh\"auser Verlag, Basel, 2007.

\bibitem{DadPen}
M.~Dadarlat and U.~Pennig.
\newblock A {D}ixmier--{D}ouady theory for strongly self-absorbing
  {$C^*$}-algebras.
\newblock {\em arXiv:1302.4468, to appear in J. Reine Angew. Math}.

\bibitem{DanErp}
C.~Daenzer and E.~Van~Erp.
\newblock T-{D}uality for {L}anglands {D}ual {G}roups.
\newblock {\em arXiv:1211.0763}.

\bibitem{DelEmeKanMey}
I.~Dell'Ambrogio, H.~Emerson, T.~Kandelaki, and R.~Meyer.
\newblock A functorial equivariant {$K$}-theory spectrum and an equivariant
  {L}efschetz formula.
\newblock {\em arXiv:1104.3441}.

\bibitem{DonKar}
P.~Donovan and M.~Karoubi.
\newblock Graded {B}rauer groups and {$K$}-theory with local coefficients.
\newblock {\em Inst. Hautes \'Etudes Sci. Publ. Math.}, (38):5--25, 1970.

\bibitem{ElmMan}
A.~D. Elmendorf and M.~A. Mandell.
\newblock Rings, modules, and algebras in infinite loop space theory.
\newblock {\em Adv. Math.}, 205(1):163--228, 2006.

\bibitem{GeiHes}
T.~Geisser and L.~Hesselholt.
\newblock Topological cyclic homology of schemes.
\newblock In {\em Algebraic {$K$}-theory ({S}eattle, {WA}, 1997)}, volume~67 of
  {\em Proc. Sympos. Pure Math.}, pages 41--87. Amer. Math. Soc., Providence,
  RI, 1999.

\bibitem{Hig1}
N.~Higson.
\newblock A characterization of {$KK$}-theory.
\newblock {\em Pacific J. Math.}, 126(2):253--276, 1987.

\bibitem{Hig2}
N.~Higson.
\newblock Algebraic {$K$}-theory of stable {$C\sp *$}-algebras.
\newblock {\em Adv. in Math.}, 67(1):140, 1988.

\bibitem{HSS}
M.~Hovey, B.~Shipley, and J.~Smith.
\newblock Symmetric spectra.
\newblock {\em J. Amer. Math. Soc.}, 13(1):149--208, 2000.

\bibitem{JoaKHom}
M.~Joachim.
\newblock {$K$}-homology of {$C^\ast$}-categories and symmetric spectra
  representing {$K$}-homology.
\newblock {\em Math. Ann.}, 327(4):641--670, 2003.

\bibitem{KaroubiConj}
M.~Karoubi.
\newblock {$K$}-th\'eorie alg\'ebrique de certaines alg\`ebres d'op\'erateurs.
\newblock In {\em Alg\`ebres d'op\'erateurs ({S}\'em., {L}es {P}lans-sur-{B}ex,
  1978)}, volume 725 of {\em Lecture Notes in Math.}, pages 254--290. Springer,
  Berlin, 1979.

\bibitem{KarWod}
M.~Karoubi and M.~Wodzicki.
\newblock Algebraic and {H}ermitian {$K$}-theory of {$K$}-rings.
\newblock {\em Quart. J. Math.}, 64:903--940, 2013.

\bibitem{KelDG}
B.~Keller.
\newblock On differential graded categories.
\newblock In {\em International Congress of Mathematicians. Vol. II}, pages
  151--190. Eur. Math. Soc., Z\"urich, 2006.

\bibitem{HMS}
M.~Kontsevich.
\newblock Homological algebra of mirror symmetry.
\newblock In {\em Proceedings of the {I}nternational {C}ongress of
  {M}athematicians, {V}ol.\ 1, 2 ({Z}\"urich, 1994)}, pages 120--139, Basel,
  1995. Birkh\"auser.

\bibitem{KonNotes}
M.~Kontsevich.
\newblock X{I} {S}olomon {L}efschetz {M}emorial {L}ecture series: {H}odge
  structures in non-commutative geometry.
\newblock In {\em Non-commutative geometry in mathematics and physics}, volume
  462 of {\em Contemp. Math.}, pages 1--21. Amer. Math. Soc., Providence, RI,
  2008.
\newblock Notes by Ernesto Lupercio.

\bibitem{KonMot}
M.~Kontsevich.
\newblock Notes on motives in finite characteristic.
\newblock In {\em Algebra, arithmetic, and geometry: in honor of {Y}u. {I}.
  {M}anin. {V}ol. {II}}, volume 270 of {\em Progr. Math.}, pages 213--247.
  Birkh\"auser Boston, Inc., Boston, MA, 2009.

\bibitem{KriMay}
I.~K{\v{r}}{\'{\i}}{\v{z}} and J.~P. May.
\newblock Operads, algebras, modules and motives.
\newblock {\em Ast\'erisque}, (233):iv+145pp, 1995.

\bibitem{LiSGC}
X.~Li.
\newblock Semigroup {$C^*$}-algebras and amenability of semigroups.
\newblock {\em to appear in JFA, arXiv:1105.5539}.

\bibitem{LiAxb}
X.~Li.
\newblock Semigroup {$C^*$}-algebras of {$ax+b$}-semigroups.
\newblock {\em arXiv:1306.5553}.

\bibitem{MyColoc}
S.~Mahanta.
\newblock Colocalizations of noncommutative spectra and bootstrap categories.
\newblock {\em arXiv:1412.8370}.

\bibitem{MyNSHLoc}
S.~Mahanta.
\newblock Symmetric monoidal noncommutative spectra, strongly self-absorbing
  {$C^*$}-algebras, and bivariant homology.
\newblock {\em arXiv:1403.4130}.

\bibitem{KQ}
S.~Mahanta.
\newblock Higher nonunital {Q}uillen {$K'$}-theory, {$KK$}-dualities and
  applications to topological {$\mathbb T$}-dualities.
\newblock {\em J. Geom. Phys.}, 61(5):875--889, 2011.

\bibitem{MyNovikov}
S.~Mahanta.
\newblock Assembly maps with coefficients in topological algebras and the
  integral {K}-theoretic {N}ovikov conjecture.
\newblock {\em J. Homotopy Relat. Struct.}, 9(2):299--315, 2014.

\bibitem{MyTwist}
S.~Mahanta.
\newblock Twisted {$K$}-theory, {$K$}-homology, and bivariant {C}hern-{C}onnes
  type character of some infinite dimensional spaces.
\newblock {\em Kyoto J. Math.}, 54(3):597--640, 2014.

\bibitem{MMSS}
M.~A. Mandell, J.~P. May, S.~Schwede, and B.~Shipley.
\newblock Model categories of diagram spectra.
\newblock {\em Proc. London Math. Soc. (3)}, 82(2):441--512, 2001.

\bibitem{MarTabApplications}
M.~Marcolli and G.~Tabuada.
\newblock Noncommutative motives and their applications.
\newblock {\em arXiv:1311.2867}.

\bibitem{MatRos2}
V.~Mathai and J.~Rosenberg.
\newblock On mysteriously missing {$T$}-duals, {$H$}-flux and the {$T$}-duality
  group.
\newblock In {\em Differential geometry and physics}, volume~10 of {\em Nankai
  Tracts Math.}, pages 350--358. World Sci. Publ., Hackensack, NJ, 2006.

\bibitem{Mey1}
R.~Meyer.
\newblock Categorical aspects of bivariant {$K$}-theory.
\newblock In {\em {$K$}-theory and noncommutative geometry}, EMS Ser. Congr.
  Rep., pages 1--39. Eur. Math. Soc., Z\"urich, 2008.

\bibitem{Nuiten}
J.~Nuiten.
\newblock Cohomological quantization of local prequantum boundary field theory.
\newblock {\em freely available at
  http://ncatlab.org/schreiber/show/master+thesis+Nuiten}.

\bibitem{QuiNonunitalK0}
D.~Quillen.
\newblock {$K\sb 0$} for nonunital rings and {M}orita invariance.
\newblock {\em J. Reine Angew. Math.}, 472:197--217, 1996.

\bibitem{IndRepRie}
M.~A. Rieffel.
\newblock Induced representations of {$C\sp{\ast} $}-algebras.
\newblock {\em Advances in Math.}, 13:176--257, 1974.

\bibitem{RieProj}
M.~A. Rieffel.
\newblock Projective modules over higher-dimensional noncommutative tori.
\newblock {\em Canad. J. Math.}, 40(2):257--338, 1988.

\bibitem{RieNCTori}
M.~A. Rieffel.
\newblock Noncommutative tori---a case study of noncommutative differentiable
  manifolds.
\newblock In {\em Geometric and topological invariants of elliptic operators
  (Brunswick, ME, 1988)}, volume 105 of {\em Contemp. Math.}, pages 191--211.
  Amer. Math. Soc., Providence, RI, 1990.

\bibitem{RorClass}
M.~R{\o}rdam.
\newblock Classification of nuclear, simple {$C^*$}-algebras.
\newblock In {\em Classification of nuclear {$C^*$}-algebras. {E}ntropy in
  operator algebras}, volume 126 of {\em Encyclopaedia Math. Sci.}, pages
  1--145. Springer, Berlin, 2002.

\bibitem{RosCT}
J.~Rosenberg.
\newblock Continuous-trace algebras from the bundle theoretic point of view.
\newblock {\em J. Austral. Math. Soc. Ser. A}, 47(3):368--381, 1989.

\bibitem{RosComparison}
J.~Rosenberg.
\newblock Comparison between algebraic and topological {$K$}-theory for
  {B}anach algebras and {$C^*$}-algebras.
\newblock In {\em Handbook of {$K$}-theory. {V}ol. 1, 2}, pages 843--874.
  Springer, Berlin, 2005.

\bibitem{RosStringDuality}
J.~Rosenberg.
\newblock {\em Topology, {$C^\ast$}-algebras, and string duality}, volume 111
  of {\em CBMS Regional Conference Series in Mathematics}.
\newblock Published for the Conference Board of the Mathematical Sciences,
  Washington, DC, 2009.

\bibitem{SchDeloop}
M.~Schlichting.
\newblock Delooping the {$K$}-theory of exact categories.
\newblock {\em Topology}, 43(5):1089--1103, 2004.

\bibitem{SchSedano}
M.~Schlichting.
\newblock Higher algebraic {$K$}-theory.
\newblock In {\em Topics in algebraic and topological {$K$}-theory}, volume
  2008 of {\em Lecture Notes in Math.}, pages 167--241. Springer, Berlin, 2011.

\bibitem{UrsQuant}
U.~Schreiber.
\newblock Quantization via {L}inear homotopy types.
\newblock {\em arXiv:1402.7041}.

\bibitem{SchwedeBook}
S.~Schwede.
\newblock Symmetric spectra.
\newblock {\em preprint, available from the author{'}s homepage}.

\bibitem{SYZ}
A.~Strominger, S.-T. Yau, and E.~Zaslow.
\newblock Mirror symmetry is {$T$}-duality.
\newblock In {\em Winter School on Mirror Symmetry, Vector Bundles and
  Lagrangian Submanifolds (Cambridge, MA, 1999)}, volume~23 of {\em AMS/IP
  Stud. Adv. Math.}, pages 333--347. Amer. Math. Soc., Providence, RI, 2001.

\bibitem{SusWod1}
A.~A. Suslin and M.~Wodzicki.
\newblock Excision in algebraic {$K$}-theory and {K}aroubi's conjecture.
\newblock {\em Proc. Nat. Acad. Sci. U.S.A.}, 87(24):9582--9584, 1990.

\bibitem{SusWod2}
A.~A. Suslin and M.~Wodzicki.
\newblock Excision in algebraic {$K$}-theory.
\newblock {\em Ann. of Math. (2)}, 136(1):51--122, 1992.

\bibitem{Tab3}
G.~Tabuada.
\newblock Higher {$K$}-theory via universal invariants.
\newblock {\em Duke Math. J.}, 145(1):121--206, 2008.

\bibitem{TabGarden}
G.~Tabuada.
\newblock A guided tour through the garden of noncommutative motives.
\newblock In {\em Topics in noncommutative geometry}, volume~16 of {\em Clay
  Math. Proc.}, pages 259--276. Amer. Math. Soc., Providence, RI, 2012.

\bibitem{TomWin}
A.~S. Toms and W.~Winter.
\newblock Strongly self-absorbing {$C^*$}-algebras.
\newblock {\em Trans. Amer. Math. Soc.}, 359(8):3999--4029, 2007.

\bibitem{Waldhausen}
F.~Waldhausen.
\newblock Algebraic {$K$}-theory of spaces.
\newblock In {\em Algebraic and geometric topology ({N}ew {B}runswick,
  {N}.{J}., 1983)}, volume 1126 of {\em Lecture Notes in Math.}, pages
  318--419. Springer, Berlin, 1985.

\end{thebibliography}

\smallskip

\end{document}